\documentclass[11pt]{article}

\usepackage{amsxtra}
\usepackage{amsthm}
\usepackage{amsmath,amscd}
\usepackage{amssymb}
\usepackage{amsfonts}
\usepackage{supertabular}
\usepackage{pstricks}
\usepackage{pst-plot}
\usepackage{pstricks-add}
\usepackage{pst-eps}
\usepackage{makeidx}
\usepackage{bm}

\newcommand{\abs}[1]{{\mathopen\mid}{#1}{\mathclose\mid}}

\newtheorem{theorem}{Theorem}[section]
\newtheorem{lemma}[theorem]{Lemma}

\newtheorem{remark}[theorem]{Remark}

\newtheorem{conjecture}[theorem]{Conjecture}

\newtheorem{scholium}[theorem]{Scholium}

\DeclareMathOperator\PGL{PGL}
\DeclareMathOperator\LL{L}
\DeclareMathOperator\SL{SL}
\DeclareMathOperator\PSL{PSL}
\DeclareMathOperator\U{U}

\DeclareMathOperator\PU{PU}

\newcommand{\N}{{\mathbb N}}
\newcommand{\C}{{\mathbb C}}

\newcommand{\Z}{{\mathbb Z}}

\makeindex

\begin{document}

\title{The Allcock Ball Quotient}

\author
{Gert Heckman\\
Radboud University Nijmegen}

\date{To Eduard Looijenga for his 65th birthday}

\maketitle

\begin{abstract}
In this article we provide further evidence for the monstrous proposal of Daniel Allcock,
by giving a plausible but still conjectural explanation for the deflation relation 
in the Coxeter group quotient of the orbifold fundamental group.
\end{abstract}

\section{Introduction}

A simply laced Coxeter diagram is just a graph for which any two distinct nodes
are either disconnected or connected by a single bond. All Coxeter diagrams in this
paper are simply laced, and therefore we shall simply write Coxeter diagram for the
longer phrase simply laced Coxeter diagram. Standard examples are the Coxeter diagrams
of type $A_n$ with $n$ nodes labeled $1,\cdots,n$ and only successive nodes are connected,
or the Coxeter diagram of type $\tilde{A}_n$ with $(n+1)$ nodes labeled $0,1,\cdots,n$ with
successive vertices connected together with a connection from $0$ to $n$.

Deleting some nodes from a Coxeter diagram together with all bonds connected with at least one of
them gives a Coxeter subdiagram. For example the Coxeter diagram of type $\tilde{A}_n$ has
the Coxeter diagram of type $A_n$ as subdiagram by deleting the node with label $0$ and
the two bonds connected to this node. A Coxeter subdiagram of type $\tilde{A}_{m}$ in a
bigger Coxeter diagram $X_n$ is called a free $(m+1)$-gon in $X_n$.
For $p,q,r\in\N$ the Coxeter tree diagram of type $Y_{pqr}$ has $n=(p+q+r)+1$ vertices
labeled $0,1,\cdots,(p+q+r)$, with a unique triple node $0$ connected to the first
nodes of three Coxeter diagrams of types $A_p,A_q,A_r$. Of special interest are
the Coxeter diagrams of the finite type $A_n=Y_{(n-1)00}$, $D_n=Y_{(n-3)11}$ for
$n\geq4$ and $E_n=Y_{(n-4)21}$ for $n=6,7,8$ for which
\[ 1/(p+1)+1/(q+1)+1/(r+1)>1\;, \]
and of the affine type $\tilde{A}_n$, $\tilde{D}_n$ and $\tilde{E}_6=Y_{222}$,
$\tilde{E}_7=Y_{331}$, $\tilde{E}_8=Y_{521}$ for which
\[ 1/(p+1)+1/(q+1)+1/(r+1)=1\;. \]
The Coxeter tree diagram of type $\tilde{D}_n$ has a unique quadruple node for $n=4$
and two triple nodes for $n\geq5$ whose deletion gives a disjoint union of $4$ Coxeter
diagrams of type $A_1$ and one of type $A_{n-5}$.

If $X_n$ is some Coxeter diagram with $n$ vertices the Artin group $A(X_n)$ is by definition
a groups with generators $T_i$ for each node $i$ of $X_n$ and relations
\[ T_iT_j=T_jT_i\;,\;T_iT_jT_i=T_jT_iT_j \]
if either $i$ and $j$ are disconnected or connected respectively.
In the former case $T_i$ and $T_j$ commute and in the latter case they braid.
The quotient group of $A(X_n)$ by the quadratic relations
\[ T_i^2=1 \]
is called the Coxeter group $W(X_n)$ of type $X_n$. For a connected Coxeter
diagram $X_n$ the group $W(X_n)$ is finite precisely for the finite type diagrams.
For the affine type diagrams $\tilde{X}_n$ (with $X=A,D,E$) the Coxeter group $W(\tilde{X}_n)$
has a free Abelian subgroup of rank $n$ with finite quotient $W(X_n)$.
In this case the quotient map
\[ W(\tilde{X}_n)\rightarrow W(X_n) \]
is called deflation, and for $X_n=A_n$ the diagram $\tilde{X}_n=\tilde{A}_n$ is also called
a free $(n+1)$-gon and one also speaks of deflation of the free $(n+1)$-gon.

If the connected Coxeter diagram of some type $X_n$ is neither of finite nor of affine type
then the Coxeter group $W(X_n)$ is of exponential growth. However for some special Coxeter
diagrams the group $W(X_n)$ has a remarkable finite quotient with a fairly simple presentation.

Label the generators of the Artin group $A(\tilde{E}_6)$ by
$a,b_1,b_2,b_3,c_1,c_2,c_3$ with $a$ the generator corresponding
to the triple node, with $c_1,c_2,c_3$ the three generators corresponding to the three
extremal nodes and $b_i$ the generator that braids with $a$ and $c_i$ for $i=1,2,3$.
The element
\[ s=ab_1c_1ab_2c_2ab_3c_3 \]
is called the spider element.
The next remarkable result is due to Ivanov and Norton \cite{Ivanov 1992},\cite{Norton 1992}.

\begin{theorem}[Ivanov and Norton]\label{555 Ivanov--Norton theorem}
The group $W(Y_{555})$ modulo the spider relation $s^{10}=1$ is equal to the wreath product
$M\wr2=(M\times M)\rtimes S_2$ of the Fisscher--Griess monster simple group $M$ with the two elements group.
\end{theorem}

Conway and Simons showed that by increasing the number of generators this presentation
takes a simpler form \cite{Conway--Simons 2001}. Let $I_{26}$ be the incidence graph of
the projective plane $\mathbb{P}^2(3)$ over a field of $3$ elements. The nodes are the
points and the lines of the projective plane, and two nodes are connected if they are incident.
The Coxeter diagram $Y_{555}$ is a maximal subtree of $I_{26}$.

\begin{theorem}[Conway and Simons]\label{555 Conway--Simons theorem}
The bimonster $M\wr2$ is obtained from the Coxeter group $W(I_{26})$ by deflating all
free $12$-gons in $I_{26}$.
\end{theorem}

We shall denote $\omega=(-1+\sqrt{-3})/2$ and $\theta=\omega-\overline{\omega}=\sqrt{-3}$.
Let $\mathcal{E}=\mathbb{Z}+\mathbb{Z}\omega$ be the ring of Eisenstein integers.
An Eisenstein lattice $L$ is a free $\mathcal{E}$-module of finite rank with a Hermitian form
$\langle \cdot,\cdot\rangle$ on $L$ with $\langle\lambda,\mu\rangle\in\theta\mathcal{E}$.
A vector $\varepsilon\in L$ with norm $\langle\varepsilon,\varepsilon\rangle=3$ is called a root in $L$.
The triflection
\[ t_{\varepsilon}(\lambda)=\lambda+(\omega-1)\frac{\langle\lambda,\varepsilon\rangle}{\langle\varepsilon,\varepsilon\rangle}\varepsilon \]
with root $\varepsilon$ is an order three complex reflection leaving $L$ invariant.

A bipartite Coxeter diagram of some type $X_n$ has the additional property that
the $n$ nodes are coloured black or white, such that bonds only connect black and white nodes.
For a Coxeter tree diagram such a bipartition is always possible, and for the incidence
diagram $I_{26}$ one just colours points black and lines white. With a bipartite Coxeter diagram
$X_n$ we associate an Eisenstein lattice $L(X_n)$ with basis $\varepsilon_i$ indexed by the nodes.
The Hermitian form is defined by
\[ \langle\varepsilon_i,\varepsilon_i\rangle=3\;,\;\langle\varepsilon_i,\varepsilon_j\rangle=0\;,\;\langle\varepsilon_p,\varepsilon_l\rangle=\theta \]
for all $i$, for all disconnected $i\neq j$ and for all connected black $p$ and white $l$.
We denote by $\U(L(X_n))$ the group of unitary automorphisms of the Eisenstein lattice $L(X_n)$.
It is easily checked that the map
\[ A(X_n)\rightarrow\U(L(X_n))\;,\;T_i\mapsto t_{\varepsilon_i} \]
extends to a Hermitian representation of the Artin group $A(X_n)$ on the Eisenstein lattice $L(X_n)$.
In fact, for a Coxeter tree diagram this is just the reflection representation
of the Hecke algebra of type $X_n$ with parameter $q=-\omega$ as constructed
by Curtis, Iwahori and Kilmoyer \cite{Curtis--Iwahori--Kilmoyer 1975}.

The Hermitian form for the Eisenstein lattice $L(Y_{555})$ has a two dimensional kernel,
and the quotient lattice is the Lorentzian Eisenstein lattice $L(Y_{544})$ of rank $14$.
We shall abbreviate $L=L(Y_{544})$ and call it the Allcock lattice.
The triflection representation gives a natural homomorphism
\[ A(Y_{555})\rightarrow\U(L) \]
which Allcock \cite{Allcock 2000} and Basak \cite{Basak 2006} have shown to be surjective.
The Allcock lattice $L$ defines in a standard way a complex hyperbolic ball
\[ \mathbb{B}=\mathbb{B}(L)=\mathbb{P}(\{z\in\C\otimes L;\langle z,z\rangle<0\}) \]
with a proper discontinuous action of $\Gamma=\PU(L)$. The ball quotient
\[ \mathbb{B}/\Gamma=\mathbb{B}/\Gamma(L) \]
is the remarkable Allcock ball quotient. For a root $\varepsilon\in L$ the hyperball
\[ \mathbb{P}(\{z\in\C\otimes L;\langle z,z\rangle<0,\langle z,\varepsilon\rangle=0\}) \]
is called the mirror for the root $\varepsilon$, and we write $\mathbb{B}^{\circ}$ for the
complement in $\mathbb{B}$ of all the mirrors. Basak has shown that all mirrors are
conjugate under $\Gamma$, and so
\[ \mathbb{B}^{\circ}/\Gamma=\mathbb{B}^{\circ}/\Gamma(L) \]
is the ball quotient complement of an irreducible Heegner divisor $\Delta$ \cite{Basak 2006}.
In his monstrous proposal Allcock made a remarkable conjecture \cite{Allcock 2006}.

\begin{conjecture}[Allcock]\label{Allcock conjecture}
The quotient of the orbifold fundamental group
\[ G=\Pi_1^{\mathrm{orb}}(\mathbb{B}^{\circ}/\Gamma) \]
by the normal subgroup $N$ generated by the squares of the meridians is the bimonster $M\wr2$.
By a meridian is meant a small loop in $\mathbb{B}^{\circ}/\Gamma$ that encircles $\Delta$
once positively at a generic point of $\Delta$.
\end{conjecture}

The original evidence for Allcock was rather modest and based on the occurrence the $Y_{555}$
diagram both in the Ivanov--Norton theorem and in his description of the lattice $L$.
Further evidence has been supplied by Basak with the following theorem \cite{Basak 2006}.

\begin{theorem}[Basak]\label{Basak 26 theorem}
The Hermitian form of the Eisenstein lattice $L(I_{26})$ has a kernel of dimension $12$
and the quotient of $L(I_{26})$ by this kernel is equal to the Allcock lattice $L$.
\end{theorem}

This is a remarkable observation, but the proof is straightforward. For $l$ the index of a white node ($l$ a line) put
\[ \delta_l=-\theta\varepsilon_l+\sum_{p\sim l}\varepsilon_p \]
with $p\sim l$ meaning that the corresponding nodes are connected ($p$ a point on $l$). Then an easy verification yields
\[ \langle\delta_l,\varepsilon_q\rangle=0\;,\;\langle\delta_l,\varepsilon_m\rangle=\theta \]
for all for all black nodes $q$ and white nodes $m$. Just distinguish $q$ on $l$ or not on $l$, and $m$ equal $l$ or not equal $l$.
Hence $\delta_l-\delta_m$ is a null vector for any two white nodes $l,m$, and these vectors span the kernel of dimension $12$.

The quotient of the triflection representation yields a homomorphism
\[ A(I_{26})\rightarrow\U(L) \]
which a fortiori is surjective. By definition the orbifold fundamental group of
$\mathbb{B}^{\circ}/\Gamma$ gives rise to an exact sequence
\[ 1\rightarrow\Pi_1(\mathbb{B}^{\circ})\rightarrow G\stackrel{\pi}\rightarrow\Gamma\rightarrow1 \]
and Basak proved the following result and made the following conjecture \cite{Basak 2012}.

\begin{theorem}[Basak]
There exists a natural homomorphism 
\[ \psi:A(I_{26})\rightarrow G \] 
whose composition with $\pi:G\rightarrow\Gamma=\mathrm{PU}(L)$ is the triflection homomorphism 
$A(I_{26})\rightarrow\Gamma$ discussed above.
\end{theorem}

Basak makes a convenient choice of base point $w_0\in\mathbb{B}^{\circ}$, which he calls the Weyl point.
He shows that there are exactly $26$ mirrors in $\mathbb{B}$ at minimal distance from $w_0$.
The loop in $G$ starting at $w_0$ along the shortest geodesic to such a mirror, making a third turn near the mirror and
continuing geodesically to the image $t_iw_0$ is denoted by $T_i$. Using a computer algorithm
Basak shows that these $T_i$ satisfy the braid relations of the incidence diagram $I_{26}$.

\begin{conjecture}[Basak]
The homomorphism $\psi:A(I_{26})\rightarrow G$ is surjective.
\end{conjecture}

Another piece of evidence for this conjecture is that the triflection homomorphism $A(I_{26})\rightarrow \Gamma$ is surjective.

Let $\mathbb{B}(V)$ be the real hyperbolic ball of dimension $13$ through $w_0$ containing these $26$ geodesics departing from $w_0$.
Each of the $26$ mirrors intersects $\mathbb{B}(V)$ in a real hyperball.
If $P\subset\mathbb{B}(V)$ is the hyperbolic polytope bounded by these $26$ hyperballs,
then $P$ is an acute angled convex polytope of finite volume by the Vinberg criterion.
Based on the analogy with the Deligne--Mostow ball quotient we are inclined to believe 
that the following conjecture holds.

\begin{conjecture}
The interior of $P$ in $\mathbb{B}(V)$ is contained in $\mathbb{B}^{\circ}$.
\end{conjecture}

The homomorphism $\psi:A(I_{26})\rightarrow G$ constructed by Basak induces a homomorphism $\varphi:W(I_{26})\rightarrow G/N$
with $N$ the normal subgroup of $G$ generated by the squares of the meridians. Our conjecture
that the interior of the polytope $P$ does not meet any mirrors can be used to show that for each free $12$-gon
in $I_{26}$ the homomorphism $\varphi$ factorizes through the deflation of the corresponding subgroup $W(\tilde{A}_{11})$. 
Hence $\varphi:W(I_{26})\rightarrow G/N$ factorizes through the bimonster $M\wr2$ by the Conway--Simons theorem.
This provides a good deal of evidence for the conjecture of Allcock.

I would like to thank the Mathematisches Forschungsinstitut at Oberwolfach for the kind hospitality while part of this work was done.
I also like to thank Tathagata Basak and Bernd Souvignier for useful comments and discussions.
This paper is dedicated to Eduard Looijenga for his $65$th birthday in gratitude for the friendship and the mathematics.

\section{Theorem of Deligne--Mostow}

Let $\mathcal{M}_{0,12}$ be the moduli space of $12$ ordered points on a curve $P$ of genus $0$
and let $\mathcal{M}_{0,12}/S_{12}$ be the moduli space of $12$ unordered points on $P$.
If $z\in\mathcal{M}_{0,12}/S_{12}$ is represented by a set of $12$ distinct complex points
$\{z_1,\cdots,z_{12}\}$ then the curve
\[C_z: y^6=\prod_1^{12}(x-z_i)\]
has genus $25$. The cyclic group $C_6$ of order $6$ acts on $C_z$ by multiplication on $y$
with quotient $P$. The ramified covering $C_z\rightarrow P$ has two intermediate coverings:
an order three covering $C_z/C_2\rightarrow P$ of genus $10$ and an order two hyperelliptic
covering $C_z/C_3\rightarrow P$ of genus $5$. This yields epimorphisms
\[ J(C_z)\rightarrow J(C_z/C_2)\;,\;J(C_z)\rightarrow J(C_z/C_3) \]
for the Jacobians, and leaves us with an Abelian variety $\mathrm{Prym}(C_z)$ with a
primitive action of $C_6$. Deligne and Mostow showed that the period map
\[ \mathrm{Per_{DM}}:\mathcal{M}_{0,12}/S_{12}\rightarrow\mathcal{A}_{10}^{\mathrm{PEL}}\]
is an isomorphism onto a ball quotient minus an irreducible Heegner divisor
\cite{Deligne--Mostow 1986},\cite{Mostow 1986}. The range of the period map denotes the period domain
of suitably Polarized Abelian varieties of genus $10$ with suitable Endomorphism and suitable Level structures.
To the best of my knowledge this viewpoint of Shimura PEL-theory has not been worked out
in the present example of Deligne--Mostow. See however the unpublished thesis by Casselman
for a partial attempt \cite{Shimura 1964},\cite{Casselman 1966}.

Let $L^n=L(A_n)$ be the Eisenstein lattice associated with the Coxeter diagram $A_n$
in the notation of the previous section. The lattice $L^n$ is positive definite for $n=1,2,3,4$.
The Hermitian form on the lattice $L^5$ has a one dimensional kernel with quotient $L^4$.
For $n=6,7,8,9,10$ it is Lorentzian, and the Hermitian form on $L^{11}$ has a one dimensional kernel
with quotient $L^{10}$. Moreover the Hermitian form on the Eisenstein lattice
$L(\tilde{A}_{11})$ has a two dimensional kernel with quotient again $L^{10}$.
We shall call the Lorentzian Eisenstein lattice $L^{10}$ the Deligne--Mostow lattice.
This explains why we have three compatible triflection representations
\[ A(A_{10}),A(A_{11}),A(\tilde{A}_{11})\longrightarrow\U(L^{10}) \]
on the Deligne--Mostow lattice $L^{10}$. Note that the Artin group $A(A_n)$
is just the original Artin braid group $\mathrm{Br}_{(n+1)}$ on $(n+1)$ strands
from $\C$. Likewise the Artin group $A(\tilde{A}_n)$ is the braid group on
$(n+1)$ strands from $\C^{\times}$.

We can construct a ball quotient $\mathbb{B}/\Gamma(L^{10})$ of dimension $9$
associated with the Deligne--Mostow lattice $L^{10}$ in the same way as was done
for the Allcock lattice $L$ in the previous section. Likewise the complement of all
mirrors is denoted $\mathbb{B}^{\circ}(L^{10})$ with quotient
$\mathbb{B}^{\circ}/\Gamma(L^{10})$. The Deligne--Mostow period map in explicit form
\[ \mathrm{Per_{DM}}:\mathcal{M}_{0,12}/S_{12}\rightarrow\mathbb{B}^{\circ}/\Gamma(L^{10}) \]
is an isomorphism of orbifolds. Their proof is geometric in nature, and requirs a careful
analysis on the left geometric and the right arithmetic side. The period map extends to
an isomorphism from the Hilbert--Mumford compactification via geometric invariant theory
for the group $\SL(2,\C)$ on degree $12$ binary forms onto the Baily--Borel compactification
of the ball quotient. The stable locus where no more than $5$ points collide is mapped
onto the ball quotient. The minimal strictly semistable locus is a single point with the collision
into two groups of $6$ points which corresponds to the unique cusp of the ball quotient
in the Baily--Borel compactification.

\section{Theorem of Couwenberg}

Consider the complex vector space $\mathcal{V}_5=\{z=(z_1,\cdots,z_5)\in\C^5;\sum z_i=0\}$
with the reflection representation of the symmetric group $S_5$.
The Coxeter group $S_5=W(A_4)$ has standard generators $s_i$ of order two ($i=1,\cdots,4$),
and together with the braid relations this is the Coxeter presentation of $S_5$.
The elementary symmetric functions $\sigma_2,\cdots,\sigma_5$ of degrees $2,\cdots,5$
are a basis for the ring of invariant polynomials. The discriminant polynomial
\[ D(\sigma_2,\cdots,\sigma_5)=\prod_{i\neq j}(z_i-z_j) \]
is the square of the product of the $10$ mirror equations,
and $D=\ast\sigma_5^4+\cdots$ is an explicit polynomial in $\sigma_2,\cdots,\sigma_5$.

Because the Hermitian form on the Eisenstein lattice $L^4$ is positive definite
the group $\U(L^4)$ is finite. Coxeter has shown that the triflection representation
\[ A(A_4)\rightarrow\U(L^4) \]
is surjective, and the cubic relations $t_i^3=1$ together with the braid relations give
a presentation of $\U(L^4)$. His proof was by computer verification \cite{Coxeter 1967}.

By the Chevalley theorem the ring of invariant polynomials on $\C\otimes L^4$ is a polynomial
algebra on four homogeneous generators, whose degrees are computed to be $12,18,24,30$.
There are $40$ mirrors and the discriminant is the cube of the product of the $40$ mirror equations.
Orlik and Solomon have shown that the generating homogeneous invariants can be chosen in such a way,
that the discriminant polynomial has the exact same expression as the discriminant polynomial
$D(\sigma_2,\cdots,\sigma_5)$ for the symmetric group $S_5$.
Their proof was again by computer verification \cite{Orlik--Solomon 1988}.

In his thesis Couwenberg has explained these results in a geometrically meaningful way
\cite{Couwenberg 1994}. One can think of his proof as the statement that the period map
\[ \mathrm{Per_C}:\mathcal{V}_5/S_5\rightarrow(\C\otimes L^4)/\U(L^4) \]
is an isomorphism of orbifolds. The Couwenberg period map is defined in terms of
similar Appell--Lauricella hypergeometric functions associated with a configuration
of $6$ points on a curve $P$ of genus $0$, one point with multiplicity $7$ (say at $\infty$)
together with $5$ unordered points  with multiplicity $1$ (on the affine line). Whereas the Deligne--Mostow
period map is related to the geometric invariant theory of the semistable points for binary forms
of degree $12$ the Couwenberg period map is related to the unstable points in the null cone.
Therefore
\[ \Pi_1^{\mathrm{orb}}((\C\otimes L^4)^{\circ}/\U(L^4))=\Pi_1(\mathcal{V}_5^{\circ}/S_5)=\mathrm{Br}_5 \]
is just the Artin braid group on 5 strands. Note that the orbifold fundamental group
and the ordinary fundamental group are the same by standard Coxeter group theory.
We arrive at a similar conclusion as before.

\begin{scholium}\label{scholium Couwenberg}
The quotient of the orbifold fundamental group
\[ G(L^4)=\Pi_1^{\mathrm{orb}}((\C\otimes L^4)^{\circ}/\U(L^4))=\mathrm{Br}_5 \]
by the subgroup generated by the squares of the meridians is the symmetric group $S_5$.
\end{scholium}

The group $S_5$ is just the Galois group of the ramified covering
\[ \mathcal{V}_5\rightarrow\mathcal{V}_5/S_5 \]
for the natural action of $S_5$.

Couwenberg obtained similar results for $S_{n+1}=W(A_n)$ acting on $\mathcal{V}_{n+1}$
and $A(A_n)\rightarrow\U(L^n)$ for $n=1,2,3,4$. The finite groups $\U(L^n)$ have $1,4,12,40$ mirrors and are the
triflection groups $\mathrm{ST}m$ for $m=3,4,25,32$ in the Shephard--Todd list \cite{Shephard--Todd 1954},\cite{Coxeter 1967}. 

\section{The orbifold fundamental group $\Pi_1^{\mathrm{orb}}(\mathcal{M}_{0,n}/S_n)$}

The orbifold fundamental group of $\mathcal{M}_{0,n}/S_n$ has been described by Looijenga
as a quotient of the affine Artin group $A(\tilde{A}_{n-1})$ with explicit relations \cite{Looijenga 1998} as follows.
Let $X$ be $\mathbb{C}^{\times}$, $\mathbb{C}$ or $\mathbb{P}=\mathbb{C}^{\times}\sqcup\{0,\infty\}$,
and let us denote by $X(n)$ the configuration space of (unordered) subsets of $X$ of cardinality $n$.
The braid group of $X$ with $n$ strands $\mathrm{Br}_n(X)$ is the fundamental group of $X(n)$.
The latter requires the choice of a base point and so is only defined up to conjugacy.
The group $\mathrm{Homeo}(X)$ of homeomorphism of $X$ acts also on $X(n)$.
The image of $\Pi_1(\mathrm{Homeo}^0(X),1)$ in $\mathrm{Br}_n(X)$ is a normal subgroup,
and the quotient shall be referred to as the braid class group $\mathrm{BrCl_n}(X)$ on $n$ strands in $X$.

First consider the case $X=\mathbb{C}^{\times}$. Take as base point $\sqrt[n]{1}$ the set of $n$th roots of $1$.
There are two special elements $R$ and $T$ in $\mathrm{Br}_n(\mathbb{C}^{\times})$:
$R$ is given by the loop of the rotation of $\sqrt[n]{1}$ over $\exp(2\pi it/n)$ for $t\in[0,1]$,
while T is represented by the loop that leaves all elements of $\sqrt[n]{1}$ in place
except $1$ and $\exp(2\pi i/n)$ which are interchanged by a counterclockwise half turn
along the circle with center $[1+\exp(2\pi i/n)]/2$ and radius $|1-\exp(2\pi i/n)|/2$ (say $n\geq5$).
These two elements generate $\mathrm{Br}_n(\C^{\times})$, but in order to get a more useful presentation
it is better to enlarge the number of generators by putting $T_k=R^kTR^{-k}$ for $k\in\mathbb{Z}/n\mathbb{Z}$.
The elements $T_k$ satisfy the affine Artin relations
\[ T_kT_{k+1}T_k=T_{k+1}T_kT_{k+1}\;,\;T_kT_l=T_lT_k \]
for all $k,l\in\mathbb{Z}/n\mathbb{Z}$ with $k-l\neq\pm1$, and together with the obvious relations
\[ RT_kR^{-1}=T_{k+1} \]
this gives a presentation of $\mathrm{Br}_n(\mathbb{C}^{\times})$ with generators
$R,T_0,\cdots,T_{n-1}$. The element $R^n$ comes from a loop in $\mathbb{C}^{\times}\subset\mathrm{Homeo}^0(\mathbb{C}^{\times})$.
Hence $R^n$ dies in $\mathrm{BrCl}_n(\mathbb{C}^{\times})$ and in fact $\mathrm{BrCl}_n(\mathbb{C}^{\times})$
is obtained from $\mathrm{Br}_n(\mathbb{C}^{\times})$ by imposing the single extra relation $R^n=1$.

Next consider the case $X=\mathbb{C}$. It is easy to check that the elements
$R,T_0,T_1,\cdots, T_{n-1}$ satisfy in $\mathrm{Br}_n(\mathbb{C})$ the additional relations
\[ R=T_1T_2\cdots T_{n-1}=T_2\cdots T_{n-1}T_0=\cdots=T_0T_1\cdots T_{n-2} \]
by filling in the origin $0$. For example, for $n=12$ the picture

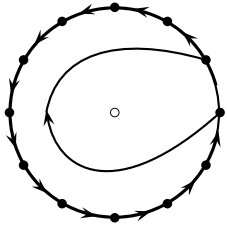
\begin{center}
\psset{unit=0.7mm}
\begin{pspicture}*(-30,-25)(30,25)

\psset{arrowscale=1,arrows=->}
\pscircle(0,0){20}
\psarc[arrowsize=1.5mm](0,0){20}{15}{50}
\psarc[arrowsize=1.5mm](0,0){20}{45}{80}
\psarc[arrowsize=1.5mm](0,0){20}{75}{110}
\psarc[arrowsize=1.5mm](0,0){20}{105}{140}
\psarc[arrowsize=1.5mm](0,0){20}{135}{170}
\psarc[arrowsize=1.5mm](0,0){20}{165}{200}
\psarc[arrowsize=1.5mm](0,0){20}{195}{230}
\psarc[arrowsize=1.5mm](0,0){20}{225}{260}
\psarc[arrowsize=1.5mm](0,0){20}{255}{290}
\psarc[arrowsize=1.5mm](0,0){20}{285}{320}
\psarc[arrowsize=1.5mm](0,0){20}{315}{350}
\psbezier[arrowsize=1.5mm]{->}(20,0)(0,-18)(-11,-11)(-13,0.5)
\psbezier[arrowsize=0.1mm](-13,0)(-11,11)(0,15)(17.32,10)

\psdot[dotstyle=Bo](0,0)
\psdot(20,0)
\psdot(17.32,10)
\psdot(10,17.32)
\psdot(0,20)
\psdot(-10,17.32)
\psdot(-17.32,10)
\psdot(-20,0)
\psdot(-17.32,-10)
\psdot(-10,-17.32)
\psdot(0,-20)
\psdot(10,-17.32)
\psdot(17.32,-10)

\end{pspicture}
\end{center}
shows that the loop $T_1\cdots T_{11}$ is homotopic to $R$ if the origin is filled in.
This gives the familiar presentation of $\mathrm{Br}_n(\mathbb{C})$ with generators
$T_1,\cdots,T_{n-1}$ and the usual Artin relations
\[ T_kT_{k+1}T_k=T_{k+1}T_kT_{k+1}\;,\;T_lT_m=T_mT_l \]
for $k,k+1,l,m\in\{1,\cdots,n-1\}$ and $l-m\neq\pm1$.
For each $\sigma\in S_{n-1}$ the Garside element $R^n=(T_{\sigma(1)}\cdots T_{\sigma(n-1)})^n$ in $\mathrm{Br}_n(\mathbb{C})$
is independent of the permutation $\sigma\in S_{n-1}$, and its square $R^{2n}$ generates the center of $\mathrm{Br}_n(\mathbb{C})$
for $n\geq3$ \cite{Deligne 1972}.

Finally consider the case that $X=\mathbb{P}$ is the projective line.
It is easy to check that the elements $R,T_0,T_1,\cdots,T_{n-1}$ satisfy in $\mathrm{Br}_n(\mathbb{P})$ the additional relations
\[ R=T_1T_2\cdots T_{n-1}\;,\;R^{-1}=T_{n-1}T_{n-2}\cdots T_1 \]
by filling in $0$ and $\infty$ respectively. Since $T_1T_2\cdots T_{n-1}$ and $T_{n-1}T_{n-2}\cdots T_1$ have the same $n$th power
in $\mathrm{Br}_n(\mathbb{C})$ the above relations already imply that $R^{2n}$ dies in $\mathrm{Br}_n(\mathbb{P})$.
This gives the presentation of $\mathrm{Br}_n(\mathbb{P})$ due to Fadell and van Buskirk \cite{Fadell--van Buskirk 1962}.
In the braid class group $\mathrm{BrCl}_n(\mathbb{P})$ we already have the relation $R^n=1$ from $\mathrm{BrCl}_n(\mathbb{C}^{\times})$.
Since $\mathrm{BrCl}_n(\mathbb{P})$ is the same thing as the orbifold fundamental group $\Pi_1^{\mathrm{orb}}(\mathcal{M}_{0,n}/S_n)$
we arrive at the presentation with generators $T_1,\cdots,T_{n-1}$ and relations the usual Artin relations together with
\[ T_1\cdots T_{n-2}T_{n-1}^2T_{n-2}\cdots T_1=1\;,\;(T_1T_2\cdots T_{n-1})^n=1 \]
which was obtained by Birman \cite{Birman 1975}.

Combining these results with the Deligne--Mostow period map we arrive at the following conclusion,
which should be thought of as a positive answer to the analogue of the conjecture of Allcock
for the Deligne--Mostow lattice $L^{10}$ rather than the Allcock lattice $L$.

\begin{scholium}\label{scholium Birman--Looijenga}
The quotient of the orbifold fundamental group
\[ G(L^{10})=\Pi_1^{\mathrm{orb}}(\mathbb{B}^{\circ}/\Gamma(L^{10})) \]
by the subgroup generated by the squares of the meridians is the symmetric group $S_{12}$.
\end{scholium}

The group $S_{12}$ is just the Galois group of the covering
\[ \mathcal{M}_{0,12}\rightarrow\mathcal{M}_{0,12}/S_{12} \]
for the natural action of $S_{12}$.

\section{Acute angled polytopes in real hyperbolic space}

Let $V$ be a real vector space of finite dimension $n+1$ with a symmetric bilinear form 
$\langle\cdot,\cdot\rangle$ of Lorentzian signature $(n,1)$. The set 
\[ \mathbb{B}(V)=\mathbb{P}(\{v\in V;\langle v,v\rangle<0\})\subset\mathbb{P}(V) \]
is a model of real hyperbolic space of dimension $n$. Suppose we have given a spanning subset 
$\{e_i;i\in I\}$ of $V$ such that its Gram matrix $G$ with entries $g_{ij}=\langle e_i,e_j\rangle$ 
satisfies $g_{ii}>0$ and $g_{ij}\leq0$ for all $i\neq j$. The set
\[ P=\mathbb{P}(\{v\in V;\langle v,v\rangle<0,\langle v,e_i\rangle\geq0\;\forall i\in I\}) \]
is called an acute angled convex polytope in the hyperbolic space $\mathbb{B}(V)$. 
We associate with this given set $\{e_i;i\in I\}$ a Coxeter diagram with nodes labeled by $I$ 
and two nodes $i,j\in I$ are connected if $g_{ij}<0$. 

For the theory of hyperbolic reflection groups such polytopes have been studied to a great extent 
by Vinberg \cite{Vinberg 1980}. A subset $J\subset I$ is called elliptic, parabolic 
or hyperbolic if the Gram matrix $G_J$ of the subset $\{e_j;j\in J\}$ is positive definite, 
positive semidefinite, or indefinite respectively. For $J\subset I$ an elliptic subset the face 
\[ P^J=\mathbb{P}(\{v\in V;\langle v,v\rangle<0,\langle v,e_i\rangle\geq0\;
\forall i\notin J,\langle v,e_j\rangle=0\;\forall j\in J\}) \]
of $P$ is not empty (by the Perron--Frobenius theorem) and of codimension equal to the cardinality 
$|J|$ of $J$. It can be shown that all faces of $P$ in $\mathbb{B}(V)$ are of this form. 
Moreover the orthogonal (geodesic) projection of $\mathbb{B}(V)$ onto the codimension $|J|$ 
hyperbolic subspace of $\mathbb{B}(V)$ containing the face $P^J$ maps the polytope $P$ 
onto its face $P^J$. 

The polytope $P$ has finite hyperbolic volume if and only if 
\[ \mathbb{P}(\{v\in V;v\neq0,\langle v,e_i\rangle\geq0\;\forall i\in I\})
\subset\mathbb{P}(\{v\in V;v\neq0,\langle v,v\rangle\leq0\}) \]
but this can be cumbersome to check in concrete examples. A subset $J\subset I$ is called 
critical if $J$ is not elliptic, but $K$ is elliptic for all proper subsets $K$ of $J$. 
Clearly critical subsets of $I$ are connected subsets of the Coxeter diagram. 
For $J$ a subset of $I$ we denote by $Z(J)$ the subset of $I$ of all nodes that 
are not connected to $J$. The next theorem is a special case of a more general result of Vinberg. 

\begin{theorem}[Vinberg] Suppose $P$ is an accute angled polytope in $\mathbb{B}(V)$ as above, 
such that each critical subset $J$ of $I$ is parabolic. Then the polytope $P$ has finite volume 
in $\mathbb{B}(V)$ if and only for each critical (parabolic) subset $J$ of $I$ the subset 
$N(J)=J\sqcup Z(J)$ is still parabolic with $G_{N(J)}$ of rank $n-1$. 
\end{theorem}

The subset $N(J)=J_1\sqcup\cdots\sqcup J_r$ in the theorem is a disjoint union of parabolic subdiagrams, 
and corresponds to an ideal vertex $P^{N(J)}$ of $P$. The local structure of $P$ near such 
an ideal vertex is a product of an interval $(0,\varepsilon)$ with a product of $r$ simplices 
of dimensions $|J_1|-1,\cdots,|J_r|-1$. 

\section{The $12$-cell of dimension 9}

The Eisenstein lattice $L^{10}=L(A_{10})$ is equal to the quotient of $L(\tilde{A}_{11})$ by its kernel. 
It has the roots $\varepsilon_i$ for $i\in\mathbb{Z}/12\mathbb{Z}$ as a generating set. 
Suppose the nodes with even index are black and with odd index are white. 
Then the Hermitian form is given by
\[\langle\varepsilon_i,\varepsilon_i\rangle=3\;,\;\langle\varepsilon_i,\varepsilon_{i+1}\rangle=(-1)^i\theta\;,\;
\langle\varepsilon_j,\varepsilon_k\rangle=0 \]
for all $i,j,k\in\mathbb{Z}/12\mathbb{Z}$ with $|j-k|\geq2$.
We shall extend scalars from the Eisenstein integers $\mathbb{Z}[\omega]$ to 
$\mathbb{Z}[\sqrt[12]{1}]$ and put 
\[ e_{2j}=i\varepsilon_{2j}\;,\;e_{2j+1}=\varepsilon_{2j+1} \]
for all $j\in\mathbb{Z}/12\mathbb{Z}$, and write $V$ for their real span.
The Gram matrix of $\{e_i;i\in\mathbb{Z}/12\mathbb{Z}\}$ becomes 
\[ \langle e_i,e_i\rangle=3\;,\;\langle e_i,e_{i+1}\rangle=-\sqrt{3}\;,\;\langle e_j,e_k\rangle=0 \]
for all $i,j,k$ with $|j-k|\geq2$. The Coxeter diagram is of type $\tilde{A}_{11}$ and 
the connected subdiagrams of type $A_n$ are elliptic for $n=1,2,3,4$, parabolic for $n=5$, 
and hyperbolic for $n=6,7,8,9,10$. The critical subdiagrams are the subdiagrams of type $A_5$,
and deleting the two adjacent nodes in the $\tilde{A}_{11}$ diagram leaves us with another 
subdiagram of type $A_5$. The rank of the Gram matrix of these two disjoint $A_5$ diagrams is $8$, 
which is the rank of $L^{10}$ minus $2$. The conditions of the theorem of Vinberg are therefore 
satisfied and we conclude that the acute angled polytope
\[ P=\mathbb{P}(\{v\in V;\langle v,v\rangle<0,\langle v,e_i\rangle\geq0\;\forall i\in I\}) \]
has finite volume in $\mathbb{B}(V)$. It has an isometric action the cyclic group $C_{12}$ 
of order $12$ which acts in a simply transitive way on the $12$ codimension one faces. 
Its center is called the Weyl point $w_0$ which has equal distance to all $12$ codimension one faces. 
The acute angled polytope $P$ of finite hyperbolic volume and of dimension $9$ will be called the $12$-cell. 

Suppose $n\geq4$ and we are given $0<\mu_1,\mu_2,\cdots,\mu_n<1$ with $\sum\mu_j=2$.
If $z_1<z_2<\cdots<z_n$ are $n$ successive real points and $z=(z_1,\cdots,z_n)$ then the Schwarz--Christoffel transformation
\[ t\mapsto v(z;t)=\int_{z_n}^t(s-z_1)^{-\mu_1}(s-z_2)^{-\mu_2}\cdots(s-z_n)^{-\mu_n}ds \]
maps the upper half plane $\Im t>0$ conformally onto a convex polygon with vertices 
$v_1=v(z;z_1)>0,v_2=v(z;z_2),\cdots,v_n=v(z;z_n)=0$ and interior angles $(1-\mu_j)\pi$ at $v_j$. 
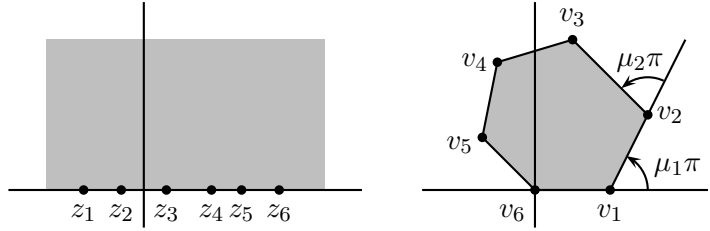
\begin{center}
\psset{unit=1mm}
\begin{pspicture}(-55,-5)(40,25)

\psline*[linecolor=lightgray](-50,0)(-13,0)(-13,20)(-50,20)
\psline(-55,0)(-8,0)
\psline(-37,-5)(-37,25)
\psdot(-45,0)
\psdot(-40,0)
\psdot(-34,0)
\psdot(-28,0)
\psdot(-24,0)
\psdot(-19,0)
\rput(-45,-3){$z_1$}
\rput(-40,-3){$z_2$}
\rput(-34,-3){$z_3$}
\rput(-28,-3){$z_4$}
\rput(-24,-3){$z_5$}
\rput(-19,-3){$z_6$}

\psline*[linecolor=lightgray](15,0)(25,0)(30,10)(20,20)(10,17)(8,7)(15,0)
\psline(15,0)(25,0)(30,10)(20,20)(10,17)(8,7)(15,0)
\psline(0,0)(38,0)
\psline(15,-5)(15,25)
\psline(30,10)(35,20)
\psset{arrowscale=1.3,arrows=->}
\psarc[linecolor=black](25,0){5}{0}{65}
\psarc[linecolor=black](30,10){5}{65}{135}
\rput(34,3){$\mu_1\pi$}
\rput(29,17){$\mu_2\pi$}
\psdot(15,0)
\psdot(25,0)
\psdot(30,10)
\psdot(20,20)
\psdot(10,17)
\psdot(8,7)
\rput(12,-3){$v_6$}
\rput(25,-3){$v_1$}
\rput(33,10){$v_2$}
\rput(20,23){$v_{3}$}
\rput(7,17){$v_{4}$}
\rput(5,6){$v_{5}$}

\end{pspicture}
\end{center}
The directed edge functions 
\[ w_j=w_j(z)=\int_{z_j}^{z_{j+1}}(s-z_1)^{-\mu_1}(s-z_2)^{-\mu_2}\cdots(s-z_n)^{-\mu_n}ds \]
satisfy $w_j=v_{j+1}-v_j$ and are called Lauricella $F_D$ hypergeometric functions of the variable $z$.
If we put $\omega_j=\exp{\pi i(\mu_1+\cdots+\mu_j)}$ then the edge lengths $l_j=\overline{\omega}_jw_j$ 
are positive real numbers (or functions of $z$) and satisfy the two linear relations
\[ \sum\omega_jl_j(z)=\sum\overline{\omega}_jl_j(z)=0 \]
making the span $V$ of the vectors $l=(l_1,\cdots,l_n)$ a real vector space of dimension $(n-2)$.

The cone $V_+=\{l\in V;l_j>0\;\forall j\}$ gets identified with the space of all such polygons
with vertices $v_1>0,v_2,\cdots,v_n=0$ and edge lengths $l_j$ from $v_j$ to $v_{j+1}$, and is called the 
polygon space of type $\mu=(\mu_1,\cdots,\mu_n)$. The spanning vector space $V$ carries a natural
Lorentzian inner product for which the norm $\langle l,l\rangle$ of $l\in V_+$ 
is equal to minus the area of the corresponding polygon. The Hermitian extension to
the complexification $\mathbb{C}\otimes_{\mathbb{R}}V$ is a monodromy invariant Lorentzian Hermitian form on
the space of Lauricella functions with parameter $\mu$. For proofs and further details we refer to the discussion
of the Lauricella $F_D$ function by Couwenberg in his thesis \cite{Couwenberg 1994}.

The parameter $\mu=(1/6,\cdots,1/6)$ is the relevant example. The set $V_+$ is identified with
the space of $12$-gons with vertices $v_1>0,v_2,\cdots,v_{12}=0$ and all interior angles equal to $5\pi/6$. 
The interior $\mathbb{P}(V_+)$ of the $12$-cell $P$ is just the space of such $12$-gons 
up to a positive scale factor. The central Weyl point $w_0$ in $P$ at equal distance to all $12$ 
codimension one faces corresponds to the regular $12$-gon.

\begin{scholium}\label{scholium 12-cell}  
The interior of the $12$-cell $P$ of dimension $9$ is contained in the mirror complement 
$\mathbb{B}^{\circ}(L^{10})$ of the Deligne--Mostow ball. The Weyl point $w_0$ in $P$
lies at equal distance to all $12$ codimension one faces. The cyclic group $C_{12}$
of order $12$ acts on $P$ by isometries leaving $w_0$ fixed.
\end{scholium}

\section{The Coxeter diagram $I_{26}$}

The projective plane $\mathbb{P}^2(3)$ over a field of $3$ elements has 13 points and 13 lines.
The incidence diagram $I_{26}$ has $26$ nodes of which $13$ are marked bold (the points) and 
$13$ hollow (the lines) with index $i$ taking values $1,2,3$.
A thin bond in the figure below indicates that the two end nodes are incident if their indices coincide, 
while a thick bond indicates that the end nodes are incident if their indices differ. So a thick bond 
represents altogether $6$ different bonds, and a thin bond just $3$. The diagram $I_{26}$ has valency $4$.
The group of diagram automorphisms of $I_{26}$ preserving the marking of the nodes is the group $\LL_3(3)=\PGL_3(3)$ 
of order $5,616=2^4\cdot3^3\cdot13$. The group $\LL_3(3).2$ of order $11,232$ obtained by adjoining an outer automorphism 
of projective duality between points and lines acts as group of automorphisms of the unmarked diagram $I_{26}$.

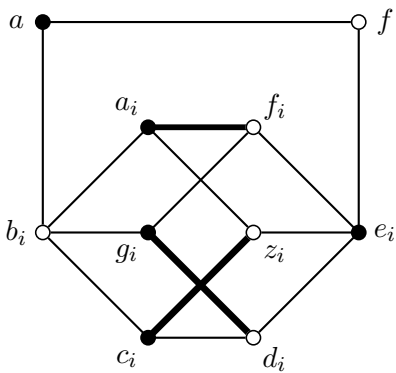
\begin{center}
\psset{unit=0.7mm}
\begin{pspicture}*(-40,-30)(40,45)

\psline(-30,0)(-30,40)
\psline(30,0)(30,40)
\psline(-30,40)(30,40)
\psline(-30,0)(-10,20)
\psline(-30,0)(-10,0)
\psline(-30,0)(-10,-20)
\psline(30,0)(10,20)
\psline(30,0)(10,0)
\psline(30,0)(10,-20)
\psline(-10,-20)(10,-20)
\psline[linewidth=1.1](-10,20)(10,20)
\psline[linewidth=1.3](-10,0)(10,-20)
\psline[linewidth=1.3](-10,-20)(10,0)
\psline(-10,20)(10,0)
\psline(-10,0)(10,20)

\psdot[dotsize=3,dotstyle=Bo](-30,0)
\psdot[dotsize=3](-10,0)
\psdot[dotsize=3,dotstyle=Bo](10,0)
\psdot[dotsize=3](30,0)
\psdot[dotsize=3](-10,20)
\psdot[dotsize=3,dotstyle=Bo](10,20)
\psdot[dotsize=3](-10,-20)
\psdot[dotsize=3,dotstyle=Bo](10,-20)
\psdot[dotsize=3](-30,40)
\psdot[dotsize=3,dotstyle=Bo](30,40)

\rput(-35,40){$a$}
\rput(35,40){$f$}
\rput(-35,0){$b_i$}
\rput(35,0){$e_i$}
\rput(-14,24){$a_i$}
\rput(14,24){$f_i$}
\rput(-14,-24){$c_i$}
\rput(14,-24){$d_i$}
\rput(-14,-4){$g_i$}
\rput(14,-4){$z_i$}

\end{pspicture}
\end{center}

Note that the subdiagram with nodes $ab_ic_id_ie_if_i$ (all $i$) by deleting the remaing nodes $fa_ig_iz_i$ (all $i$) and all bonds
connected to these remaining nodes is the $Y_{555}$ diagram, which is just a maximal subtree of $I_{26}$. Deleting the triple node $a$
of this $Y_{555}$ diagram shows that the $I_{26}$ diagram has a subdiagram of type $3A_5$. Adjoining $a_3$ and deleting $b_3$
shows that $I_{26}$ also has a subdiagram of type $A_4+\tilde{A}_{11}$ with the $4$ nodes $c_3d_3e_3f_3$ making $A_4$
and the $12$ nodes $ab_1c_1d_1e_1f_1a_3f_2e_2d_2c_2b_2$ making $\tilde{A}_{11}$.
The $\tilde{A}_{11}$ subdiagram is also called a free $12$-gon. The remaining $10$ nodes $a_1a_2b_3fg_iz_i$ (all $i$)
are each connected to both this $A_4$ subdiagram and this $\tilde{A}_{11}$ subdiagram. Hence $A_4$ and $\tilde{A}_{11}$
determine each other uniquely as the maximal disjoint complementary subdiagram in $I_{26}$.   
In our previous notation $Z(A_4)=\tilde{A}_{11}$ and $Z(\tilde{A}_{11})=A_4$.
Observe also that $ab_i$ and $d_ie_if_iz_i$ (all $i$) yields a subdiagram of $I_{26}$ of type $4D_4$.

\section{The $26$-cell of dimension $13$}

The set $I=\mathcal{P}\sqcup\mathcal{L}$ of $26$ vertices of the Coxeter diagram $I_{26}$
splits as a disjoint union of the $13$ points and the $13$ lines of $\mathbb{P}^2(3)$.
If $\varepsilon_i$ is the generating set of $L$ with Gram matrix
\[ \langle\varepsilon_i,\varepsilon_i\rangle=3\;,\;\langle\varepsilon_j,\varepsilon_k\rangle=0\;,\;\langle\varepsilon_p,\varepsilon_l\rangle=\theta \]
for all $i$, for all disconnected $j\neq k$ and for all connected $p\in\mathcal{P}$ and $l\in\mathcal{L}$
then we introduce a new set $\{e_i\}$ simply by
\[ e_p=i\varepsilon_p, e_l=\varepsilon_l \]
fior $p\in\mathcal{P}$ and $l\in\mathcal{L}$.
The Gram matrix of $e_i$ becomes the real symmetric matrix
\[ \langle e_i,e_i\rangle=3\;,\;\langle e_i,e_j\rangle=0\;,\;\langle e_j,e_k\rangle=-\sqrt{3}\]
for all $i,j,k$ with $i\neq j$ disconnected and $j\neq k$ connected.
If for each line $l\in\mathcal{L}$ we put $d_l=\sqrt{3}e_l+\sum_{p\sim l}e_p$ then it is easy to check that
\[ \langle d_l,e_q\rangle=0\;,\;\langle d_l,e_m\rangle=-\sqrt{3} \]
for all $q\in\mathcal{P}$ and $m\in\mathcal{L}$. 
Hence $d_l-d_m$ is a null vector for all $l,m\in\mathcal{L}$ and we conclude that the real vector 
space $V$ spanned by the vectors $\{e_i;i\in I\}$ becomes a Lorentzian space of dimension $14$.

The acute angled hyperbolic polytope
\[ P=\mathbb{P}(\{v\in V;\langle v,v\rangle<0,\langle v,e_i\rangle\geq0\;\forall i\in I\}) \]
in $\mathbb{B}(V)$ will be called the $26$-cell. It has dimension $13$ and finite hyperbolic volume.
Indeed it is easy to check that the critical subdiagrams are the connected parabolic diagrams of type 
$A_5$ or $D_4$. Since $N(A_5)=3A_5$ and $N(D_4)=4D_4$ are both parabolic and have both Gram matrices 
of rank $12$ this follows from the theorem of Vinberg.

The $26$-cell $P$ has two natural vertices $w_{\mathcal{P}}$ perpendicular to all $e_p$ with $p\in\mathcal{P}$
and $w_{\mathcal{L}}$ perpendicular to all $e_l$ with $l\in\mathcal{L}$. The midpoint $w_0$ on the geodesic from
$w_{\mathcal{P}}$ to $w_{\mathcal{L}}$ is called the Weyl point. The group $L_3(3).2$ of diagram automorphisms
of the unmarked Coxeter diagram $I_{26}$ acts a group of isometries of $P$ leaving the Weyl point $w_0$ fixed.
The $26$-cell $P$ has two inequivalent ideal vertices of the above types $3A_5$ and $4D_4$. 
The next conjecture is the analogue of Scholium{\;\ref{scholium 12-cell}} for the $26$-cell $P$.

\begin{conjecture}\label{26-cell conjecture}
The interior of the $26$-cell $P$ is the connected component of $\mathbb{B}(V)\cap\mathbb{B}^{\circ}(L)$ containing $w_0$?
In other words, the interior of $P$ does not meet any mirror of the complex Allcock ball $\mathbb{B}(L)$?
\end{conjecture}

Partial results towards this conjecture are due to Basak \cite{Basak 2012}. He shows that the $26$ mirrors supported by the
codimension one faces of the $26$-cell $P$ are exactly those mirrors in the Allcock ball $\mathbb{B}(L)$ that are 
nearest to the Weyl point $w_0$. The real subbal $\mathbb{B}(V)\subset\mathbb{B}(L)$ supported by $P$ contains 
all $26$ shortest geodesics from $w_0$ to these nearest mirrors, and this characterizes $\mathbb{B}(V)$. 
In particular for each vertex $i$ of $I_{26}$ the geodesic from $w_0$ to the orthogonal projection $w_i$ of $w_0$ 
on the codimension one face $P^i$ of $P$ does not meet any mirror in $\mathbb{B}(L)$ before it reaches $w_i$.

Basak defines a curve $\gamma_i$ in $\mathbb{B}^{\circ}$ with begin point the Weyl point $w_0$ and
end point $t_i(w_0)$. Here $t_i$ is the triflection with eigenvalue $\omega$ leaving the codimension 
one face $P^i$ fixed. The curve $\gamma_i$ is almost the geodesic from $w_0$ to $w_i$ and then continues 
geodesically to $t_i(w_0)$. However this curve hits the mirror supported by $P^i$ at $w_i$ and so instead 
shortly before arriving at $w_i$ it makes a one third turn in the complex line through $w_0,w_i,t_i(w_0)$.
The curve $\gamma_i$ defines the meridian element $T_i$ of $\Pi_1^{\mathrm{orb}}(\mathbb{B}^{\circ}/\Gamma,w_0)$.

For $i,j$ two different nodes of $I_{26}$ Basak proves the Artin braid relations
\[ T_iT_jT_i=T_jT_iT_j\;,\;T_iT_j=T_jT_i \]
in case $i,j$ are connected or disconnected respectively along the following lines.
Let $w_{ij}$ be the orthogonal projection of $w_0$ on the codimension two face $P^{ij}$ of $P$.
Basak shows that the interior of the convex hull of the $4$ points $w_0,w_i,w_j,w_{ij}$ does not
meet any mirror of $\mathbb{B}(L)$. The curve $\gamma_i$ can be continuously deformed in $\mathbb{B}^{\circ}(L)$
to a curve $\gamma_{ij}$ going geodesically from $w_0$ to $w_{ij}$ and shortly before arriving at $w_{ij}$
making a one third turn around the mirror supported by $P^{i}$. Likewise $\gamma_j$ can be deformed to $\gamma_{ji}$.
The braid relation for the two corresponding meridians is a local relation of the mirror arrangement near $w_{ij}$
and follows from the work of Couwenberg as described in Section 3, or by giving the explicit homotopy as Basak did.
If $i,j$ are connected then four mirrors pass through $P^{ij}$ while in case $i,j$ are disconnected only two 
orthogonal mirrors pass trough $P^{ij}$.

The group $\LL_3(3).2$ of diagram automorphisms of the unmarked diagram $I_{26}$ acts by isometries on the $26$-cell $P$.
The Weyl point $w_0$ is a fixed point for this action. The infinitesimal action of $\LL_3(3).2$ on the tangent space of
$\mathbb{B}(V)$ at $w_0$ decomposes as a direct sum of a one dimensional representation (the line through $w_{\mathcal{P}}$
and $w_{\mathcal{L}}$) and an irreducible representation of dimension $12$ on the orthogonal complement.
This is the smallest dimensional irreducible representation of $\LL_3(3).2$ that is nontrivial on $\LL_3(3)$.

Let $J$ be a subdiagram of $I_{26}$ of type $A_4$. Any two such subdiagrams are conjugated under $\LL_3(3).2$
and so we can assume that $J$ consists of the nodes $c_3d_3e_3f_3$ in the notation of Section $7$.
The complementary subdiagram $Z(J)$ obtained by deleting all nodes of $J$ and those connected to $J$
contains the $12$ nodes $ab_1c_1d_1e_1f_1a_3f_2e_2d_2c_2b_2$ and is of type $\tilde{A}_{11}$.
The face $P^J$ of $P$ of codimension $4$ is just the $12$-cell of dimension $9$ as discussed in Section $6$.
We denote by $\mathbb{B}(U)$ the real hyperbolic space supported by $P^J$, viewed as subspace of
the real hyperbolic space $\mathbb{B}(V)$ supported by $P$. The subgroup of $\LL_3(3).2$ preserving the face $P^J$
is the cyclic group $C_{12}$ of order $12$, with generator $R$ permuting the nodes $ab_1c_1d_1e_1f_1a_3f_2e_2d_2c_2b_2$
in cyclic way, while also permuting $c_3f_3$, $d_3e_3$, $fg_1z_3g_2$ and $b_3a_1z_2g_3z_1a_2$ in cyclic way.

\begin{lemma}
Any positive definite Eisenstein lattice of rank $5$ containing $L^4$ as a primitive sublattice
and spanned by $L^4$ and a complementary root is of the form $L^4\oplus L^1$.
\end{lemma}

\begin{proof}
By assumption the lattice has a root basis $\varepsilon_1,\cdots,\varepsilon_4,\varepsilon_5$
with the first four vectors the standard basis of $L^4$. If we assume that
\[ \langle\varepsilon_1,\varepsilon_5\rangle=x\theta\;,\;\langle\varepsilon_2,\varepsilon_5\rangle=y\theta\;,\;
\langle\varepsilon_3,\varepsilon_5\rangle=z\theta\;,\;\langle\varepsilon_4,\varepsilon_5\rangle=w\theta \]
then the determinant of the Gram matrix (divided by $9$)
is easily found to be
\begin{eqnarray*}
3-x\overline{x}-w\overline{w}-2(y\theta-x)(\overline{y\theta-x})-2(z\theta+w)(\overline{z\theta+w})+ \\
-\theta(y\theta-x)(\overline{z\theta+w})+\theta(z\theta+w)(\overline{y\theta-x})
\end{eqnarray*}
with $x,y,z,w\in\mathcal{E}$. Since
\[ 2a\overline{a}+2b\overline{b}+\theta a\overline{b}-\theta b\overline{a}=
(a-b\omega)(\overline{a-b\omega})+(a+b\omega)(\overline{a+b\omega}) \]
the above expression becomes
\[ 3-x\overline{x}-w\overline{w}-(a-b\omega)(\overline{a-b\omega})-(a+b\omega)(\overline{a+b\omega}) \]
with $a=y\theta-x,b=z\theta+w$. This expression should be positive, and so
\[ x\overline{x}\leq1\;,\;w\overline{w}\leq1\;,\;
(a-b\omega)(\overline{a-b\omega})\leq1\;,\;(a+b\omega)(\overline{a+b\omega})\leq1 \]
and their sum is at most $2$, so at least two terms are $0$.

If $x=w=0$ then $a=y\theta,b=z\theta$ which implies $y=z=0$.
Similarly if $x=0,a=b\omega$ then $a=b=0$ which in turn implies $y=z=w=0$.
Finally if $a=b\omega=-b\omega$ then $a=b=0$ and so $x=y=z=w=0$.
\end{proof}

Hence the complexification $\mathbb{B}(L_{\mathrm{DM}})$ of $\mathbb{B}(U)$ in the Allcock ball $\mathbb{B}(L_{\mathrm{A}})$
is the intersection of $40$ mirrors in $\mathbb{B}(L_{\mathrm{A}})$, and $\mathbb{B}(U)=\mathbb{B}(V)\cap\mathbb{B}(L_{\mathrm{DM}})$.
All other mirrors in $\mathbb{B}(L_{\mathrm{A}})$ intersecting $\mathbb{B}(L_{\mathrm{DM}})$ do so in a perpendicular way. 
The local structure of the $26$-cell $P$ near its face $P^J$ is a product of $P^J$ with a real simplicial chamber 
$P_J$ of dimension $4$ of the group $\U(L^4)$ corresponding to $5$ ordered points on $\mathbb{R}$ with zero sum,
as discussed in Section $3$.

Let $J$ be the given subset of $I_{26}$ of type $A_4$ with complement $Z(J)$ of type $\tilde{A}_{11}$.
Let $w_J$ be the orthogonal projection on the face $P^J$ of the Weyl point $w_0$ of $P$.
The point $w_J$ is the central point of $P^J$ corresponding to the regular $12$-gon in the Deligne--Mostow picture.
For $j\in Z(J)$ let $w_{jJ}$ be the projection on $w_0$ on the face $P^{jJ}$ (with $jJ$ standing for $\{j\}\sqcup J$), 
which is the same as the orthogonal projection of $w_J$ on the the codimension one face $P^{jJ}$ of $P^J$.
Now the above conjecture implies that the curve $\gamma_j$ can be continuously deformed to a curve $\gamma_{jJ}$,
which is a curve $\tilde{\gamma}_j$ in the tubular neighborhood of $\mathbb{B}_{\mathrm{DM}}$ in $\mathbb{B}_{\mathrm{A}}$ 
with base point a nearby point $\tilde{w}_J$ of $w_J$ conjugated by a geodesic from this nearby point to $w_0$. 

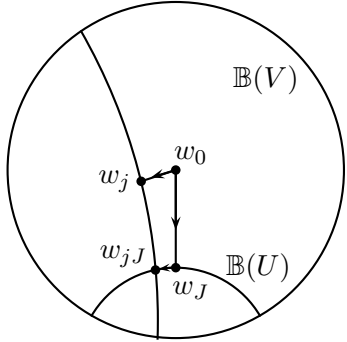
\begin{center}
\psset{unit=1mm}
\begin{pspicture}*(-25,-25)(25,25)

\pscircle(0,0){22.5}
\psdot(0,0)
\rput(2,2){$w_0$}
\rput(12,12){$\mathbb{B}(V)$}
\rput(11,-13){$\mathbb{B}(U)$}

\psarc(-80,-20){77.62}{-1.9}{29.8}
\psdot(-2.7,-13.3)
\psdot(-4.6,-1.5)
\rput(-8.1,-1.5){$w_j$}
\rput(-7,-11.5){$w_{jJ}$}
\psline(0,0)(-4.6,-1.5)
\psline[arrows=->,arrowscale=1.2](0,0)(-3.33,-1)

\psarc(0,-25.98){12.99}{30}{150}
\psdot(0,-12.99)
\rput(2,-16){$w_J$}
\psarc[arrows=->,arrowscale=1](0,-25.98){12.99}{90}{100}

\psline[arrows=->,arrowscale=1.2](0,0)(0,-8)
\psline(0,0)(0,-12.99)

\end{pspicture}
\end{center}
Indeed the desired homotopy is obtained using the orthogonal projection of $P$ onto its face $P^J$.
Under the identification of $Z(J)$ with $\mathbb{Z}/12\mathbb{Z}$ the meridian elements 
$T_i\in\Pi_1^{\mathrm{orb}}(\mathbb{B}^{\circ}/\Gamma,w_0)$ for $i\in\mathbb{Z}/12\mathbb{Z}$
satisfy the Artin braid relations
\[ T_iT_{i+1}T_i=T_{i+1}T_iT_{i+1}\;,\;T_iT_j=T_jT_i \]
for $i-j\neq\pm1$ of the affine Artin group of type $\tilde{A}_{11}$.

The inclusion map of the face $P^J$ of $P$ gives rise to a holomorphic map from the Deligne--Mostow ball quotient
$\mathbb{B}/\Gamma(L_{\mathrm{DM}})$ to the Allcock ball quotient $\mathbb{B}/\Gamma(L_{\mathrm{A}})$. 
This map is an immersion, but not an injection, since the image of $\mathbb{B}/\Gamma(L_{\mathrm{DM}})$ 
in $\mathbb{B}/\Gamma(L_{\mathrm{A}})$ has triple self intersection along a one dimensional ball quotient, 
which is isomorphic to the modular curve $\mathbb{H}_+/\mathrm{PSL}(2,\mathbb{Z})$.
Let $\mathbb{N}^{\circ}$ be the normalisation of a small tubular neighbourhood of 
$\mathbb{B}/\Gamma(L_{\mathrm{DM}})$ inside the mirror complement $\mathbb{B}^{\circ}/\Gamma(L_{\mathrm{A}})$.
Then we have a fiber bundle
\[ \mathbb{N}^{\circ}\rightarrow \mathbb{B}^{\circ}/\Gamma(L_{\mathrm{DM}}) \]
with fiber a small ball around the origin in $(\mathbb{C}\otimes L^4)^{\circ}$ modulo $\mathrm{U}(L^4)$.
This gives rise to an exact homotopy sequence
\[ 1\rightarrow\Pi_1^{\mathrm{orb}}((\mathbb{C}\otimes L^4)^{\circ}/\mathrm{U}(L^4))\rightarrow
\Pi_1^{\mathrm{orb}}(\mathbb{N}^{\circ})\rightarrow\Pi_1^{\mathrm{orb}}(\mathbb{B}^{\circ}/\Gamma(L_{\mathrm{DM}}))\rightarrow1 \]
and taking the quotient by squares of meridians we conclude that the group $\Pi_1^{\mathrm{orb}}(\mathbb{N}^{\circ})$
modulo squares of meridians is isomorphic to $S_5\times S_{12}$. Indeed the only action of $S_{12}$ by automorphisms on $S_5$ 
is the trivial action. Hence the image of the subgroup generated by the $T_i$ for $i\in\mathbb{Z}/12\mathbb{Z}$ under
the homomorphism $\varphi:W(I_{26})\rightarrow G/N$ is a factor group of $S_{12}$. In other words, the free $12$-gons
are deflated in $G/N$. Therefore the monstrous proposal of Allcock is a consequence of Conjecture{\;\ref{26-cell conjecture}}.
The following remark I learned from Eduard Looijenga.

\begin{remark}
One can show that the orbifold fundamental group of the image of $\mathbb{N}^\circ$ in $\mathbb{B}^{\circ}/\Gamma(L_A)$ 
is obtained from that of $\mathbb{N}^\circ$ by means of an $\mathrm{HNN}$ extension (after Higman, Neumann and Neumann) 
\cite{Serre 1980}. To be precise, the fiber orbifold fundamental group 
$\Pi_1^{\mathrm{orb}}(\mathrm{Fiber})\subset \Pi_1^{\mathrm{orb}}(\mathbb{N}^\circ)$ also appears as the image of an embedding 
$h: \Pi_1^{\mathrm{orb}}(\mathrm{Fiber})\to \Pi_1^{\mathrm{orb}}(\mathrm{Base})$ and the $\mathrm{HNN}$ extension in question simply adds an 
extra generator $t$ to $\Pi_1^{\mathrm{orb}}(\mathbb{N}^{\circ})$ subject to the relation that conjucagy with $t$ restricted to 
$\Pi_1^{\mathrm{orb}}(\mathrm{Fiber})$ is a lift of $h$. So if we subsequently divide out by the (normal) subgroup generated 
by the squares of the meridians, then we get an $\mathrm{HNN}$ extension of $S_5\times S_{12}$ relative to the standard inclusion of $S_5$ 
in the second factor. Note that Conway and Pritchard \cite{Conway--Pritchard 1992} characterize the bimonster as the smallest 
quotient of this $\mathrm{HNN}$ extension, which still contains  $S_5\times S_{12}$ and is not isomorphic to $S_{17}$. 
\end{remark}

\section{Final remarks}

The results of this paper beg for the existence of a suitable moduli space $\mathcal{M}$ and a period map
\[ \mathrm{Per}:\mathcal{M}\rightarrow\mathbb{B}^{\circ}/\Gamma(L) \]
as in previous sections. In fact the conjecture of Allcock implies the existence of an quasi projective orbifold $\mathcal{N}$
together with an action of the bimonster $M\wr2$ such that $\mathcal{N}/(M\wr2)=\mathcal{M}$.
All I can say for the moment is that the marked space $\mathcal{N}$ must be of a spectacular geometric complexity 
because of the structure and size of the bimonster.

There is a variation on the theme of this paper with Eisenstein integers $\mathcal{E}$ replaced by Gauss integers
$\mathcal{G}=\Z+\Z i$ and Gauss lattice automorphism groups generated by tetraflections (order four complex reflections). 
There is a ball quotient, also found by Allcock \cite{Allcock 2000}, associated with Gauss lattices of type $Y_{333}$ as 
the maximal subtree of the incidence graph $I_{14}$ of the projective plane $\mathbb{P}^2(2)$ over the field of two elements. 
The Coxeter diagram $I_{14}$ has the form (with index $i$ taking values $1,2,3$)
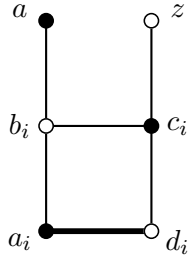
\begin{center}
\psset{unit=0.7mm}
\begin{pspicture}*(-20,-27
)(20,27)

\psline(-10,20)(-10,0)
\psline(-10,0)(-10,-20)
\psline(-10,0)(10,0)
\psline(10,20)(10,0)
\psline(10,0)(10,-20)
\psline[linewidth=1.1](-10,-20)(10,-20)

\psdot[dotsize=3](-10,20)
\psdot[dotsize=3](-10,-20)
\psdot[dotsize=3,dotstyle=Bo](-10,0)
\psdot[dotsize=3](10,0)
\psdot[dotsize=3,dotstyle=Bo](10,-20)
\psdot[dotsize=3,dotstyle=Bo](10,20)

\rput(-15,22){$a$}
\rput(15,22){$z$}
\rput(-15,-22){$a_i$}
\rput(15,-22){$d_i$}
\rput(-15,0){$b_i$}
\rput(15,0){$c_i$}


\end{pspicture}
\end{center}
with similar meaning as for the $I_{26}$ diagram. A thin bond indicates that the two nodes are connected
if indices coincide, and a thick bond indicates that the two nodes are connected if indices differ.
The automorphism group of the unmarked (forget coloring of the nodes) diagram is the group 
$L_3(2).2$ with $L_3(2)$ the simple group of order $168=2^3\cdot3\cdot7$.

The Lorentzian Gauss lattice $L(Y_{322},\mathcal{G})$ has rank $8$ and appears as quotient modulo kernel 
of the degenerate Gauss lattices $L(Y_{333},\mathcal{G})$ and $L(I_{14},\mathcal{G})$. 
The Hermitian form on $L(I_{14},\mathcal{G})$ is defined by
\[ \langle\varepsilon_j,\varepsilon_j\rangle=2\;,\;\langle\varepsilon_j,\varepsilon_k\rangle=0\;,\;\langle\varepsilon_p,\varepsilon_l\rangle=1+i \]
for all $j$, for all disconnected $j\neq k$ and for all connected black $p$ and white $l$ (points $p$ on a line $l$).
If $G(Y_{322},\mathcal{G})$ is the orbifold fundamental group of $\mathbb{B}^{\circ}/\Gamma(Y_{322},\mathcal{G})$
then we get natural homomorphisms
\[ A(I_{14})\rightarrow G(Y_{322},\mathcal{G})\rightarrow\U(L(Y_{322},\mathcal{G})/(1+i)L(Y_{322},\mathcal{G}))=\mathrm{O}^-_8(2).2 \]
by taking the quotient of the squares of the meridians, in agreement with the presentation of $\mathrm{O}^-_8(2).2$ 
due to Simons \cite{Simons 2001}. This analogy provides a further piece of evidence for the conjecture of Allcock.
Even in this simpler situation a modular interpretation for the regular part $\mathbb{B}^{\circ}/\Gamma(L(Y_{322},\mathcal{G}))$ 
of this ball quotient is lacking. The hyperball quotient $\mathbb{B}^{\circ}/\Gamma(L(E_7),\mathcal{G}))$ is 
isomorphic to the moduli space of smooth quartic plane curves (a result due to Kondo \cite{Kondo 2000}) 
and this is the regular part of the mirror discriminant $\Delta$ in $\mathbb{B}/\Gamma(L(Y_{322},\mathcal{G}))$.

\noindent
Gert Heckman, Radboud University Nijmegen, P.O. Box 9010,
\newline
6500 GL Nijmegen, The Netherlands (E-mail: g.heckman@math.ru.nl)

\end{document}